\documentclass[a4paper,oneside,11pt, reqno, amsfonts]{amsproc}
\usepackage{graphicx}
\usepackage{float}
\usepackage{pstricks}
\usepackage{amscd}
\usepackage{amsmath}
\usepackage{amsxtra}
\usepackage[T1]{fontenc}
\usepackage[utf8]{inputenc}
\usepackage{lipsum}
\usepackage{tikz}
\usetikzlibrary{arrows}
\usetikzlibrary{calc}

\usepackage{amsfonts,amssymb,amsmath,amsthm}
\usepackage{url}
\usepackage{enumerate}

\newcommand{\strutstretchdef}{\newcommand{\strutstretch}{\vphantom{\raisebox{1pt}{$\big($}\raisebox{-1pt}{$\big($}}}}
\theoremstyle{plain}
\newtheorem{theorem}{Theorem}[section]

\newtheorem{proposition}[theorem]{Proposition}
\newtheorem{corollary}[theorem]{Corollary}

\theoremstyle{definition}
\newtheorem{definition}[theorem]{Definition}

\newtheorem{example}[theorem]{Example}

\theoremstyle{remark}
\newtheorem{remark}[theorem]{Remark}

\numberwithin{equation}{section}

\newlength{\struh}
\newlength{\textminustop}
\strutstretchdef
\newcommand{\ran}{\mbox{ran}}

\usepackage{mathrsfs}
\usepackage{epsfig}
\usepackage[active]{srcltx}

\newcommand{\ncom}{\newcommand}
\ncom{\bq}{\begin{equation}}
\ncom{\eq}{\end{equation}}
\ncom{\beqn}{\begin{eqnarray*}}
\ncom{\eeqn}{\end{eqnarray*}}
\ncom{\beq}{\begin{eqnarray}}
\ncom{\eeq}{\end{eqnarray}}
\ncom{\nno}{\nonumber}
\ncom{\rar}{\rightarrow}
\ncom{\Rar}{\Rightarrow}
\ncom{\noin}{\noindent}
\ncom{\bc}{\begin{centre}}
\ncom{\ec}{\end{centre}}
\ncom{\sz}{\scriptsize}
\ncom{\rf}{\ref}
\ncom{\sgm}{\sigma}
\ncom{\Sgm}{\Sigma}
\ncom{\dt}{\delta}
\ncom{\Dt}{Delta}
\ncom{\lmd}{\lambda}
\ncom{\Lmd}{\Lambda}
\ncom{\eps}{\epsilon}
\ncom{\pcc}{\stackrel{P}{>}}
\ncom{\dist}{{\rm\,dist}}
\ncom{\sspan}{{\rm\,span}}
\ncom{\im}{{\rm Im\,}}
\ncom{\sgn}{{\rm sgn\,}}
\ncom{\ba}{\begin{array}}
\ncom{\ea}{\end{array}}
\ncom{\eop}{\hfill{{\rule{2.5mm}{2.5mm}}}}
\ncom{\eoe}{\hfill{{\rule{1.5mm}{1.5mm}}}}
\ncom{\eof}{\hfill{{\rule{1.5mm}{1.5mm}}}}
\ncom{\hone}{\mbox{\hspace{1em}}}
\ncom{\htwo}{\mbox{\hspace{2em}}}
\ncom{\hthree}{\mbox{\hspace{3em}}}
\ncom{\hfour}{\mbox{\hspace{4em}}}
\ncom{\hsev}{\mbox{\hspace{7em}}}
\ncom{\vone}{\vskip 2ex}
\ncom{\vtwo}{\vskip 4ex}
\ncom{\vonee}{\vskip 1.5ex}
\ncom{\vthree}{\vskip 6ex}
\ncom{\vfour}{\vspace*{8ex}}
\ncom{\norm}{\|\;\;\|}
\ncom{\integ}[4]{\int_{#1}^{#2}\,{#3}\,d{#4}}
\ncom{\inp}[2]{\langle{#1},\,{#2} \rangle}
\ncom{\Inp}[2]{\Langle{#1},\,{#2} \Langle}
\ncom{\vspan}[1]{{{\rm\,span}\#1 \}}}
\ncom{\dm}[1]{\displaystyle {#1}}

\begin{document}
\title[Bounded point evaluation]{Bounded point evaluation for operators with the wandering subspace property}

\author[S. Trivedi]{Shailesh Trivedi}

%\thanks{The second named author was partially supported by NSF Grant DMS-1302666.} \

\address{Department of Mathematics and Statistics\\
Indian Institute of Technology Kanpur, India}
%\email{akasha@iitk.ac.in}
%   \email{chavan@iitk.ac.in}
 \email{shailtr@iitk.ac.in}

%\dedicatory{Dedicated to the memory of Serguei Shimorin}

\thanks{The work of the author is supported through the Inspire Faculty Fellowship DST/INSPIRE/04/2018/000338.}



\begin{abstract}
We extend and study the notion of bounded point evaluation introduced by Williams for a cyclic operator to the class of operators with the wandering subspace property. We characterize the set $bpe(T)$ of all bounded point evaluations for an operator $T$ with the wandering subspace property in terms of the invertibility of certain projections. This result generalizes the earlier established characterization of $bpe(T)$ for a finitely cyclic operator $T$. Further, if $T$ is a left-invertible operator with the wandering subspace property, then we determine the $bpe(T)$ and the set $abpe(T)$ of all analytic bounded point evaluations for $T$. We also give examples of left-invertible operator $T$ with the wandering subspace property for which $\mathbb D\big(0, r(T')^{-1}\big) \subsetneqq abpe(T) \subseteq bpe(T)$, where $r(T')$ is the spectral radius of the Cauchy dual $T'$ of $T$.
\end{abstract}

\maketitle

\section{Introduction}

The primary motivation for the present work comes from Shimorin's analytic model for a left-invertible analytic operator \cite{S}, which facilitates us to model a left-invertible analytic operator $T$ as the operator of multiplication by the co-ordinate function on the Hilbert space $\mathscr H$ of $\ker T^*$-valued holomorphic functions defined on $\mathbb D\big(0, r(T')^{-1}\big)$, where $r(T')$ is the spectral radius of the Cauchy dual $T'$ of $T$. This analytic model gives that the open disc $\mathbb D\big(0, r(T')^{-1}\big)$ is contained in the set of all bounded point evaluations on $\mathscr H$. On the other hand, the example \cite[Example 2]{CT} shows that the disc $\mathbb D\big(0, r(T')^{-1}\big)$ is not optimum and may be properly contained in the set of all bounded point evaluations on $\mathscr H$ in general. Therefore, it comes very naturally to inquire about the set of all bounded point evaluations on $\mathscr H$. For this, we adapt the abstract approach of Williams \cite{W} who, motivated by the work of Trent \cite{Tr}, introduced and studied the notion of bounded point evaluation for a cyclic operator. His idea was extended by Miller-Miller-Neumann \cite{MMN} for rationally cyclic operators which was further generalized by the trio Guendafi-Mbekhta-Zerouali \cite{GMZ} for finitely multicyclic operators. Later, Mbekhta-Ourchane-Zerouali \cite{MOZ} extended and studied the notion of bounded point evaluation for rationally multicyclic operators. A contribution in the study of bounded point evaluations for a cyclic operator was also made by Bourhim-Chidume-Zerouali \cite{BCZ}.

In this paper, we extend and study the notion of bounded point evaluation for an operator with the wandering subspace property. Besides, studying the general properties of the set $bpe(T)$ of all bounded point evaluations for an operator $T$ with the wandering subspace property, we also describe $bpe(T)$ in terms of the invertibility of the restriction of specific projections (see Theorem \ref{bpe-thm}). In addition, we give the complete description of $bpe(T)$ for a left-invertible operator $T$ with the wandering subspace property. This comprises a major part of Section 2. In the third section, we study the notion of analytic bounded point evaluation and characterize the set $abpe(T)$ of all analytic bounded point evaluations for a left-invertible operator $T$ with the wandering subspace property. This, in particular, recovers the Shimorin's analytic model for a left-invertible operator with the wandering subspace property. We wind up this paper with examples of left-invertible operator $T$ with the wandering subspace property for which $\mathbb D\big(0, r(T')^{-1}\big) \subsetneqq abpe(T) \subseteq bpe(T)$.

We set below the notations to be used in the posterior sections. The set of non-negative integers, the field of real numbers and the field of complex numbers are denoted by $\mathbb N, \mathbb R$ and $\mathbb C$ respectively. The complex conjugate of a complex number $w$ is denoted by $\overline{w}$, and for a non-empty subset $F$ of $\mathbb C$, the complex conjugate of $F$ is given by $\overline{F} := \{\overline w : w \in F\}$. The interior of $F$ is denoted by $\mbox{int}\, F$. An open disc in the complex plane centred at $w$ with radius $r >0$ is denoted by $\mathbb D(w,r)$ whereas the closed disc centred at $w$ with radius $r >0$ is denoted by $\overline{\mathbb D}(w,r)$. For a polynomial $p$, $\mbox{deg} \, p$ denotes the degree of $p$. Let $\mathcal H$ be a complex separable Hilbert space. By a subspace of $\mathcal H$ we do not mean a closed subspace of $\mathcal H$. If $\mathcal M$ is a subspace of $\mathcal H$, then its closure is denoted by $\mbox{cl}\, \mathcal M$. Further, if $\mathcal M$ is a subspace of $\mathcal H$, then by $\dim \mathcal M$ we mean the Hilbert space dimension of $\mathcal M$. The orthogonal complement of a subspace $\mathcal M$ is denoted by $\mathcal M^\bot$. If $\mathcal M$ is a closed subspace of $\mathcal H$, then $P_{\mathcal M}$ stands for the orthogonal projection of $\mathcal H$ onto $\mathcal M$. For a subset $S$ of $\mathcal H$, the linear span and the closed linear span of $S$ are denoted by $\mbox{span}\{x : x \in S\}$ and $\bigvee\{x : x \in S\}$ respectively. Let $\mathcal B(\mathcal H)$ denote the algebra of bounded linear operators on $\mathcal H$. For $T \in \mathcal B(\mathcal H)$, the kernel, range and adjoint of $T$ are denoted by $\ker T$, $\ran\, T$ and $T^*$ respectively. The spectrum and the point spectrum of $T \in \mathcal B(\mathcal H)$ are respectively denoted by $\sigma(T)$ and $\sigma_p(T)$ whereas $r(T)$ stands for the spectral radius of $T$. If $T \in \mathcal B(\mathcal H)$ is left-invertible, then $T^*T$ is invertible. In this case, the Cauchy dual $T'$ of $T$ is defined as $T' := T(T^*T)^{-1}$. Further, we say that $T \in \mathcal B(\mathcal H)$ is analytic if $\bigcap_{n \in \mathbb N} T^n(\mathcal H) = \{0\}$.

\section{Bounded point evaluation: General properties}

Let $T$ be a bounded linear operator on a complex separable Hilbert space $\mathcal H$. A closed subspace $\mathcal M$ of $\mathcal H$ is said to be a {\it cyclic subspace} for $T$ if 
\beqn
\bigvee_{n \in \mathbb N} T^n(\mathcal M) = \mathcal H.
\eeqn
If $\dim \mathcal M$ is finite, then $T$ is said to be {\it finitely cyclic}. We say that $T$ has the {\it wandering subspace property} if $\ker T^*$ is a cyclic subspace for $T$. A rich supply of examples of operators with the wandering subspace property is provided by the class of operator-valued unilateral shift \cite[Theorem 4.5]{GKT}. Let $T \in \mathcal B(\mathcal H)$ have the wandering subspace property. Then for a $\ker T^*$-valued polynomial $p(z) = \sum_{n=0}^k x_n z^n$, $x_n \in \ker T^*$, we set  
\beqn
p(T) := \sum_{n=0}^k T^n x_n.
\eeqn
If $\mathcal P(\mathbb C, \ker T^*)$ denotes the vector space of all $\ker T^*$-valued polynomials, then $\{p(T) : p \in \mathcal P(\mathbb C, \ker T^*)\}$ is a dense subspace of $\mathcal H$. Thus, we have the following definition.

\begin{definition}
Let $T$ be a bounded linear operator on a complex separable Hilbert space $\mathcal H$. Suppose that $T$ has the wandering subspace property. A point $w \in \mathbb C$ is said to be a {\it bounded point evaluation for} $T$ if  there exists a positive constant $c$ such that
\beqn
\|p(w)\| \leqslant c\, \|p(T)\| \ \mbox{for all } p \in \mathcal P(\mathbb C, \ker T^*).
\eeqn
The set of all bounded point evaluations for $T$ is denoted by $bpe(T)$.
\end{definition}

The notation $bpe(T)$ was used for the set of all bounded point evaluations for a finitely cyclic operator (with respect to some cyclic subspace) \cite{GMZ}. It can be easily seen that if $T$ is finitely cyclic with the wandering subspace property, then $bpe(T)$ given by the preceding definition coincides with the $bpe(T)$ defined in \cite{GMZ}.

Let $w \in bpe(T)$. Then the map $E_w:\{p(T) : p \in \mathcal P(\mathbb C, \ker T^*)\} \rar \ker T^*$ defined by
\beqn
E_w (p(T)) = p(w), \quad p \in \mathcal P(\mathbb C, \ker T^*),
\eeqn
is well-defined, linear, surjective and continuous. Also, it is easy to see that $E_w|_{\ker T^*} = I_{\ker T^*}$. Since $\{p(T) : p \in \mathcal P(\mathbb C, \ker T^*)\}$ is a dense subspace of $\mathcal H$, $E_w$ can be extended continuously on $\mathcal H$. We denote the continuous extension of $E_w$ on $\mathcal H$ by $E_w$ itself. The following proposition studies the general properties of $E_w$.

\begin{proposition}\label{bpe-prop}
Let $T$ be a bounded linear operator on a complex separable Hilbert space $\mathcal H$. Suppose that $T$ has the wandering subspace property. Then the following statements hold:
\begin{enumerate}
\item[$(i)$] $0 \in bpe(T) \subseteq \overline{\sigma_p(T^*)}$. 
\item[$(ii)$] If $w \in bpe(T)$, then $E_w^*(\ker T^*) = \ker(T^* - \overline{w})$. Moreover, $E_w^* : \ker T^* \rar \ker(T^* - \overline{w})$ is invertible.
\item[$(iii)$] $\ker E_w = \textup{cl\,ran}(T - w)$ for each $w \in bpe(T)$.
\item[$(iv)$] For each $w \in bpe(T)$, $E_w E^*_w$ is invertible on $\ker T^*$. Moreover, $E_0 E^*_0 = I_{\ker T^*}$.
\item[$(v)$] If $w \in bpe(T)$, then $E_w T^n x = w^n E_w x$ for all $x \in \mathcal H$ and $n \in \mathbb N$.  
\end{enumerate}
\end{proposition}

\begin{proof}
The proof of (i) goes exactly along the lines of the proofs of \cite[Proposition 2.4(i),(ii)]{T} with minor modifications. To see (ii), let $w \in bpe(T)$. For $p \in \mathcal P(\mathbb C, \ker T^*)$ and $x \in \ker T^*$, we have 
\beqn
\inp{T^* E_w^* x}{p(T)} &=& \inp{E_w^*x}{T p(T)} = \inp{x}{E_w(T p(T))} = \inp{x}{w p(w)} \\
&=&  \inp{\overline{w} x}{p(w)} = \inp{\overline{w} E_w^*x}{p(T)}.
\eeqn
Since $\{p(T) : \mathcal P(\mathbb C, \ker T^*)\}$ is a dense subspace of $\mathcal H$, it follows that $T^* E_w^* x = \overline{w} E_w^*x$. Thus $E_w^*(\ker T^*) \subseteq \ker(T^* - \overline{w})$. Note that $E_w^* : \ker T^* \rar \ker(T^* - \overline{w})$ is expansive. Indeed, for each $x \in \ker T^*$, we have
\beqn
\|x\|^2 = \inp{x}{E_w x} = \inp{E_w^* x}{x} \leqslant \|E_w^* x\| \|x\|.
\eeqn
Thus $E_w^*$ is injective and has closed range. Suppose that there exists $h \in \ker(T^* - \overline{w})$ such that $\inp{E_w^* x}{h} = 0$ for all $x \in \ker T^*$. This, in turn, implies that $\inp{x}{E_w h} = 0$ for all $x \in \ker T^*$, and hence, $E_w h = 0$. Also, for any $p(z) = \sum_{n=0}^k x_n z^n \in \mathcal P(\mathbb C, \ker T^*)$, we get
\beq\label{h-pw}
\Big \langle h, \sum_{n=0}^k T^n x_n \Big \rangle = \sum_{n=0}^k \inp{T^{*n} h}{x_n} = \sum_{n=0}^k \inp{\overline{w}^n h}{x_n} = \inp{h}{p(w)}.
\eeq
Let $\{p_n\}_{n \geqslant 1}$ be a sequence of polynomials in $\mathcal P(\mathbb C, \ker T^*)$ such that $p_n(T) \rar h$ as $n \rar \infty$. Then
\beqn
\|h\|^2 = \lim_{n \rar \infty} \inp{h}{p_n(T)} \overset{\eqref{h-pw}}= \lim_{n \rar \infty} \inp{h}{p_n(w)} = \inp{h}{E_w h} = 0.
\eeqn
Thus $E_w^*(\ker T^*) = \ker(T^* - \overline{w})$ and hence (ii) stands verified.

For (iii), observe that for each $w \in bpe(T)$, we have
\beqn
\ker E_w = {\ran\, E^*_w}^\bot \overset{(ii)}= \ker(T^* - \overline{w})^\bot = \mbox{cl}\, \ran\,(T - w).
\eeqn

To see (iv), let $w \in bpe(T)$. Note that $E_w|_{\ker(T^* - \overline{w})} : \ker(T^* - \overline{w}) \rar \ker T^*$ is an invertible operator by virtue of (iii). Further, it follows from (ii) that $E_w^* : \ker T^* \rar \ker(T^* - \overline{w})$ is invertible. Thus $E_w E^*_w$ is invertible on $\ker T^*$. For the moreover part, let $x \in \ker T^*$ and $p(z) = \sum_{n=0}^k x_n z^n \in \mathcal P(\mathbb C, \ker T^*)$. Then we get
\beqn
\inp{E^*_0 x}{p(T)} = \inp{x}{E_0 p(T)} = \inp{x}{p(0)} = \inp{x}{p(T)}.
\eeqn
Since $\{p(T) : p \in \mathcal P(\mathbb C, \ker T^*)\}$ is a dense subspace of $\mathcal H$, it follows that $E^*_0 x = x$ for all $x \in \ker T^*$. This gives that $E_0 E^*_0 x = x$ for all $x \in \ker T^*$, and hence, establishes (iv).

The conclusion of (v) was derived in the proof of \cite[Proposition 2.2]{T}. Here we provide a different proof. Suppose that $w \in bpe(T)$. Then $E_w : \mathcal H \rar \ker T^*$ is continuous. Let $x \in \mathcal H$. Then for each $n \in \mathbb N$ and $y \in \ker T^*$, we have
\beqn
\inp{E_w T^n x}{y} = \inp{x}{T^{*n} E^*_w y} = \inp{x}{\overline{w}^n E^*_w y} = \inp{w^n E_w x}{y},
\eeqn
where in the second equality, we used (ii). Thus, we must have $E_w T^n x = w^n E_w x$ for each $n \in \mathbb N$. This completes the proof.
\end{proof}

A bounded linear operator $T$ on a complex separable Hilbert space $\mathcal H$ is said to be {\it circular} if for each real number $\theta$ there exists a unitary $U_\theta$ on $\mathcal H$ such that $U_\theta T = e^{i\theta} T U_\theta$. Circular operators were introduced and studied in \cite{AHHK}. Later, Gellar \cite{G} characterized all the circular operators and proved that an operator is circular if and only if it is an operator-valued shift \cite[Proposition 1]{G}. The following proposition shows that if $T$ is circular, then $bpe(T)$ has circular symmetry.

\begin{proposition}\label{circular symmetry}
Let $T$ be a bounded linear operator on a complex separable Hilbert space $\mathcal H$ with the wandering subspace property. If $T$ is circular, then the set $bpe(T)$ of all bounded point evaluations for $T$ has circular symmetry. 
\end{proposition}

\begin{proof}
Suppose that $T$ is circular. Let $\theta$ be an arbitrary but fixed real number. We first show that 
\beq\label{bpe-eq}
bpe(T) = bpe(e^{i\theta} T). 
\eeq
The proof of \eqref{bpe-eq} follows by imitating the arguments of the proof of \cite[Proposition 2.3]{T} but we include the details for the sake of completeness. To this end, let $U_\theta$ be the unitary on $\mathcal H$ such that $U_\theta T = e^{i\theta} T U_\theta$. Since $\ker T^* = \ker (e^{i\theta} T)^*$, we must have $U_\theta (\ker T^*) = \ker T^*$. Now suppose that $w \in bpe(T)$. Let $p(z) = \sum_{n=0}^k x_n z^n \in \mathcal P(\mathbb C, \ker T^*)$. Then $p(z) = \sum_{n=0}^k U_\theta y_n z^n $, $y_n \in \ker T^*$. Thus, we have 
\beqn
\|p(w)\| &=& \Big\|\sum_{n=0}^k U_\theta y_n w^n\Big\| = \Big\|\sum_{n=0}^k y_n w^n\Big\| \leqslant c \Big\|\sum_{n=0}^k T^n y_n \Big\|\\
&=& c \Big\|\sum_{n=0}^k U_\theta T^n y_n \Big\| = c \Big\|\sum_{n=0}^k e^{in\theta}T^n U_\theta y_n \Big\| = c \Big\|\sum_{n=0}^k e^{in\theta}T^n x_n \Big\|\\
&=& c \|p(e^{i\theta} T)\|.
\eeqn
 This shows that $w \in bpe(e^{i\theta} T)$, and hence, $bpe(T) \subseteq bpe(e^{i\theta} T)$. The other way inclusion can be obtained similarly. This completes the verification of \eqref{bpe-eq}.
 
 Further, suppose that $w \in bpe(T)$. Let $p(z) = \sum_{n=0}^k x_n z^n \in \mathcal P(\mathbb C, \ker T^*)$. Then $p(e^{i\theta} w) = \sum_{n=0}^k e^{in\theta} w^n x_n = \tilde p(w)$, where $\tilde p(z) = \sum_{n=0}^k y_n z^n$ with $y_n = e^{in\theta} x_n$. Thus, we get
\beqn
 \|p(e^{i\theta} w)\| = \|\tilde p(w)\| \leqslant c \Big\|\sum_{n=0}^k T^n y_n \Big\| = c \Big\|\sum_{n=0}^k e^{in\theta} T^n x_n \Big\| = c \|p(e^{i\theta} T)\|.
\eeqn
This shows that $e^{i\theta} w \in bpe(e^{i\theta} T) \overset{\eqref{bpe-eq}}= bpe(T)$, and hence, completes the proof.
\end{proof}

We are now ready to give a complete description of $bpe(T)$. The following theorem extends \cite[Theorem 4]{GMZ}. 

\begin{theorem}\label{bpe-thm}
Let $T$ be a bounded linear operator on a complex separable Hilbert space $\mathcal H$ with the wandering subspace property. Let $\{x_n : 1 \leqslant n \leqslant \dim \ker T^*\}$ be an orthonormal basis of $\ker T^*$. For $w \in \mathbb C$, let $P_w$ denote the orthogonal projection of $\mathcal H$ onto $\ker (T^* - \overline{w})$. Then the following statements are equivalent:
\begin{enumerate}
\item[$(i)$] A complex number $w$ is a bounded point evaluation for $T$.
\item[$(ii)$] $P_w|_{\ker T^*} : \ker T^* \rar \ker (T^* - \overline{w})$ is invertible.
\item[$(iii)$] There exists a sequence of vectors $\{y_n : 1 \leqslant n \leqslant \dim \ker T^*\}$ in $\ker (T^* - \overline{w})$ such that $\inp{x_i}{y_j} = \delta_{ij}\ (\mbox{the Kronecker delta})$ for all $i,j = 1, \ldots, \dim \ker T^*$.
\end{enumerate}
\end{theorem}

\begin{proof}
We first establish the equivalence of (i) and (ii). To this end, let $w \in bpe(T)$. Then it can be easily shown that  
\beqn
\mathcal H = \bigvee_{n \in \mathbb N} (T - w)^n(\ker T^*).
\eeqn
Applying $P_w$ on both sides of the above identity, we get that $\ker (T^* - \overline{w}) = P_w (\ker T^*)$. Thus, $P_w|_{\ker T^*}$ is onto. To see that $P_w|_{\ker T^*}$ is one-one, assume that $P_w x = 0$ for some $x \in \ker T^*$. Then $x \in \textup{cl\,ran}(T - w)$. Let $\{y_n\}_{n \geqslant 1}$ be a sequence in $\mathcal H$ such that $(T - w) y_n \rar x$ as $n \rar \infty$. This, together with Proposition \ref{bpe-prop}(v), gives that
\beqn
0 = E_w(T - w) y_n \rar E_w x = x \mbox{ as } n \rar \infty.
\eeqn
Thus $x = 0$.

Conversely, suppose that $P_w|_{\ker T^*} : \ker T^* \rar \ker (T^* - \overline{w})$ is invertible for some $w \in \mathbb C$. Let $p(z) = \sum_{n=0}^k x_n z^n \in \mathcal P(\mathbb C, \ker T^*)$. Then for all $y \in \ker (T^* - \overline{w})$, we have
\beqn
\inp{p(T)}{y} = \Big\langle \sum_{n=0}^k T^n x_n, y \Big\rangle = \sum_{n=0}^k \inp{T^n x_n}{y} = \sum_{n=0}^k \inp{w^n x_n}{y} = \inp{p(w)}{y}.
\eeqn
The above equality gives that $\|P_w p(T)\| = \|P_w p(w)\|$. Since $P_w|_{\ker T^*}$ is invertible, we get that
\beqn
\|p(w)\| &=& \|(P_w|_{\ker T^*})^{-1} P_w p(w)\| \leqslant \|(P_w|_{\ker T^*})^{-1}\| \|P_w p(T)\|\\
&\leqslant& \|(P_w|_{\ker T^*})^{-1}\| \|P_w\| \|p(T)\|.
\eeqn
Hence, $w \in bpe(T)$. 

To see the equivalence of (ii) and (iii), suppose that $P_w|_{\ker T^*} : \ker T^* \rar \ker (T^* - \overline{w})$ is invertible for some $w \in \mathbb C$. Let $\{h_n : 1 \leqslant n \leqslant \dim \ker T^*\}$ be an orthonormal basis of $\ker (T^* - \overline{w})$. Let $(a_{ij})_{1 \leqslant i,j \leqslant \dim \ker T^*}$ be the matrix representation of $(P_w|_{\ker T^*})^{-1}$ with respect to the bases $\{h_n : 1 \leqslant n \leqslant \dim \ker T^*\}$ and $\{x_n : 1 \leqslant n \leqslant \dim \ker T^*\}$ of $\ker (T^* - \overline{w})$ and $\ker T^*$ respectively. Set
\beqn
y_j := \sum_{k=1}^{\dim \ker T^*} a_{kj} h_k, \quad j=1, \ldots, \dim \ker T^*.
\eeqn
Then it can be easily verified that $y_j \in \ker (T^* - \overline{w})$ and $\inp{x_i}{y_j} = \delta_{ij}$ for all $i, j = 1, \ldots, \dim \ker T^*$.

Conversely, assume that (iii) holds. It was already seen in the beginning of the proof that $P_w (\ker T^*) = \ker (T^* - \overline{w})$. Thus it only remains to show that $P_w|_{\ker T^*}$ is one-one. To this end, let $P_w x = 0$ for some $x \in \ker T^*$. Then for all $j = 1, \ldots, \dim \ker T^*$, we get
\beqn
0 = \inp{P_w x}{y_j} = \inp{x}{y_j} = \Big\langle \sum_{i=1}^{\dim \ker T^*} \inp{x}{x_i} x_i, y_j \Big\rangle = \inp{x}{x_j}.
\eeqn
Hence $x=0$. This completes the proof.
\end{proof}

We characterize below $bpe(T)$ for a left-invertible operator $T$ with the wandering subspace property (cf. \cite[Proposition 5.1]{GKT}). 

\begin{theorem}\label{bpe-left-thm}
Let $T$ be a left-invertible operator on a complex separable Hilbert space $\mathcal H$ and let $T'$ denote the Cauchy dual of $T$. Suppose that $T$ has the wandering subspace property. For each $w \in \mathbb C$ and $n \in \mathbb N$, set
\beqn
\mathcal M_n := \bigvee_{k=0}^n T^k(\ker T^*) \quad \mbox{and} \quad S_n(w) := \sum_{k=0}^n \overline{w}^k T'^k.
\eeqn
Then $w \in bpe(T)$ if and only if $\displaystyle\sup_{n \in \mathbb N}\|P_{\mathcal M_n} S_n(w)|_{ker T^*}\| < \infty$. Moreover, if $w \in bpe(T)$, then $\|E_w^*\| = \displaystyle\sup_{n \in \mathbb N}\|P_{\mathcal M_n} S_n(w)|_{ker T^*}\|$.
\end{theorem}

\begin{proof}
Suppose that $w \in bpe(T)$. Then $E_w : \mathcal H \rar \ker T^*$ is continuous. Let $n \in \mathbb N$. Then for each $\ker T^*$-valued polynomial $p$ with degree up to $n$ and for each $x \in \ker T^*$, we get
\beq\label{Sn-Ew}
\inp{S_n(w) x}{p(T)} = \inp{x}{p(w)} = \inp{E_w^*x}{p(T)}. 
\eeq
Let $y \in \mathcal M_n$. Then there exists a sequence of $\ker T^*$-valued polynomials $p_m$ with degree up to $n$ such that $p_m(T) \rar y$ as $m \rar \infty$. Thus for each $x \in \ker T^*$, we get
\beqn
\inp{S_n(w) x}{y} = \lim_{m \rar \infty} \inp{S_n(w) x}{p_m(T)} \overset{\eqref{Sn-Ew}}= \lim_{m \rar \infty} \inp{E_w^*x}{p_m(T)} = \inp{E_w^*x}{y}.
\eeqn
This gives that $P_{\mathcal M_n}(S_n(w)x - E_w^*x) = 0$ for all $n \in \mathbb N$. Consequently, we obtain
\beq\label{norm-Ew}
\|P_{\mathcal M_n}S_n(w)x\| \leqslant \|E_w^*x\| \mbox{ for all } n \in \mathbb N \mbox{ and for all } x \in \ker T^*.
\eeq
Application of the uniform boundedness principle in the above inequality yields the desired conclusion.

Conversely, suppose that $w \in \mathbb C$ is such that $\sup_{n \in \mathbb N}\|P_{\mathcal M_n} S_n(w)|_{ker T^*}\| < \infty$. Let $p$ be any $\ker T^*$-valued polynomial. Then for any $n \geqslant \mbox{deg}\, p$, we have
\beqn
\|p(w)\| &=& \sup_{x \in \ker T^*, \ \|x\|=1} |\inp{x}{p(w)}| = \sup_{x \in \ker T^*, \ \|x\|=1} |\inp{S_n(w)x}{p(T)}| \\
&\leqslant& \|P_{\mathcal M_n} S_n(w)|_{ker T^*}\| \|p(T)\| \leqslant \Big(\sup_{n \in \mathbb N}\|P_{\mathcal M_n} S_n(w)|_{ker T^*}\|\Big) \|p(T)\|.
\eeqn
Thus $w \in bpe(T)$. 

For the moreover part, assume that $w \in bpe(T)$. Then \eqref{norm-Ew} implies that $\displaystyle\sup_{n \in \mathbb N}\|P_{\mathcal M_n} S_n(w)|_{ker T^*}\| \leqslant \|E_w^*\|$. To see the reverse inequality, let $y \in \mathcal H$ with $\|y\| = 1$. Then there exists a sequence $\{p_n(T)\}_{n \in \mathbb N}$ with $\|p_n(T)\| = 1$ for all $n \in \mathbb N$ such that $p_n(T) \rar y$ as $n \rar \infty$. For all $x \in \ker T^*$ with $\|x\| = 1$, we get
\beqn
|\inp{E_w^*x}{y}| &=& \lim_{n \rar \infty} |\inp{E_w^*x}{p_n(T)}| \overset{\eqref{Sn-Ew}}= \lim_{n \rar \infty} |\inp{S_{deg\, p_n}(w) x}{p_n(T)}|\\
&\leqslant& \lim_{n \rar \infty} \|P_{\mathcal M_{deg\, p_n}}S_{deg\, p_n}(w) x\| \leqslant \sup_{n \in \mathbb N}\|P_{\mathcal M_n} S_n(w)|_{ker T^*}\|.
\eeqn 
As a consequence, we get $\|E_w^*\| \leqslant \sup_{n \in \mathbb N}\|P_{\mathcal M_n} S_n(w)|_{ker T^*}\|$. This completes the proof. 
\end{proof}

The following corollary shows that the $\mbox{int}\, bpe(T)$ is non-empty for a left-invertible operator $T$ with the wandering subspace property. In fact, it shows that $bpe(T)$ contains an open disc centred at the origin which is in general bigger that $\mathbb D\big(0, r(T')^{-1}\big)$. To see this, we need the following definition: For $T \in \mathcal B(\mathcal H)$ and $x \in \mathcal H$, we set $r_T(x) : = \displaystyle \limsup_{n \rar \infty} \|T^n x\|^\frac{1}{n}$. In the literature, $r_T(x)$ is known as {\it local spectral radius} of $T$ at $x$ \cite{LN}. Clearly, $r_T(x) \leqslant r(T)$.  

\begin{corollary}\label{bpe-incl-cor}
Let $T$ be a left-invertible operator on a complex separable Hilbert space $\mathcal H$ and let $T'$ denote the Cauchy dual of $T$. Suppose that $T$ has the wandering subspace property. Then the open disc $\mathbb D(0, r)$ is contained in the $bpe(T)$, where $\displaystyle r := \inf_{x \in \ker T^*} \frac{1}{r_{T'}(x)}$. In particular, $\mathbb D\big(0, r(T')^{-1}\big) \subseteq bpe(T)$. Moreover, for each $w \in \mathbb D(0, r)$, we have $E_w^* x = \sum_{n \in \mathbb N} \overline{w}^n T'^n x$ for all $x \in \ker T^*$.  
\end{corollary}

\begin{proof}
For $x \in \ker T^*$, the radius of convergence of the power series $\sum_{n \in \mathbb N} \overline{w}^n T'^n x$ is $\big(\limsup_{n \rar \infty} \|T'^n x\|^\frac{1}{n}\big)^{-1} = \frac{1}{r_{T'}(x)}$. If we set $r := \inf_{x \in \ker T^*} \frac{1}{r_{T'}(x)}$, then it follows that if $w \in \mathbb D(0, r)$, then $\sup_{n \in \mathbb N} \|S_n(w) x\| < \infty$ for all $x \in \ker T^*$. Hence, by the uniform boundedness principle, we get that $\sup_{n \in \mathbb N} \|S_n(w)|_{\ker T^*}\| < \infty$. Thus, $\mathbb D(0, r) \subseteq bpe(T)$. For the moreover part, let $p$ be any $\ker T^*$-valued polynomial and $w \in \mathbb D(0, r)$. Then for each $n \geqslant \mbox{deg}\, p$, by \eqref{Sn-Ew}, we get $\inp{S_n(w)x}{p(T)} = \inp{E_w^* x}{p(T)}$. Taking $n \rar \infty$, we get the desired conclusion.
\end{proof}

The subspace $\mathcal M :=\bigvee_{w \in bpe(T)} E^*_w(\ker T^*)$ plays a central role in studying the function theoretic behaviour of $T$. In fact, if $\mathcal M = \mathcal H$, then it turns out that $T$ can be modelled as $\mathscr M_z$ on a reproducing kernel Hilbert space of $\ker T^*$-valued functions defined on $bpe(T)$ (see Theorem \ref{Mz}). The following proposition studies the general properties of $\mathcal M$ and also gives a dichotomy that $\mathcal M$ is either whole $\mathcal H$ or $\mathcal M$ skips an infinite dimensional subspace if $T$ is injective and analytic.   

\begin{proposition}\label{M-perp}
Let $T$ be a bounded linear operator on a complex separable Hilbert space $\mathcal H$ with the wandering subspace property. Let $\mathcal M := \bigvee_{w \in bpe(T)} E^*_w(\ker T^*)$. Then the following statements hold:
\begin{enumerate}
\item[$(i)$] $\mathcal M^\bot$ is $T$-invariant.
\item[$(ii)$] If $T$ is injective and analytic, then either $\mathcal M = \mathcal H$ or $\mathcal M^\bot$ is infinite dimensional.
\item[$(iii)$] If $T$ is circular, then $T|_{\mathcal M^\bot}$ is circular.
\item[$(iv)$] If $T$ is analytic, then $T|_{\mathcal M^\bot}$ is analytic.
\item[$(v)$] $P_{\mathcal M^\bot}(\ker T^*) \subseteq \ker(T|_{\mathcal M^\bot})^*$.
\end{enumerate}
\end{proposition}

\begin{proof}
Observe that $\mathcal M^\bot = \bigcap_{w \in bpe(T)} \ker E_w$. Let $x \in \mathcal M^\bot$. Then by Proposition \ref{bpe-prop}(v), we have $E_w T x = w E_w x= 0$ for all $w \in bpe(T)$. Thus $T x \in \mathcal M^\bot$. This proves (i).

Suppose that $T$ is injective and analytic. Let $\mathcal M^\bot$ be finite dimensional. Since $T$ is injective, it follows from (i) that $T|_{\mathcal M^\bot}$ is invertible. But then analyticity of $T$ gives that
$\mathcal M^\bot \subseteq \bigcap_{n \in \mathbb N} T^n(\mathcal H) = \{0\}.
$
This completes the verification of (ii).

To see (iii), assume that $T$ is circular. In view of Proposition \ref{bpe-prop}(iii), we have 
\beqn
\mathcal M^\bot = \bigcap_{w \in bpe(T)} \ker E_w = \bigcap_{w \in bpe(T)} \textup{cl\,ran}(T - w).
\eeqn
Let $\theta \in \mathbb R$. Since $T$ is circular, there exists a unitary $U_\theta$ on $\mathcal H$ such that $U_\theta T = e^{i\theta} T U_\theta$. Let $x \in \ran\,(T-w)$. Then $x = (T-w) y$ for some $y \in \mathcal H$. This gives that
\beqn
U_\theta x = U_\theta (T-w) y = (e^{i\theta} T - w) U_\theta y.
\eeqn
Thus $U_\theta \ran\,(T-w) \subseteq \ran\,(e^{i\theta} T-w)$. Consequently, $U_\theta \textup{cl\,ran}(T - w) \subseteq \textup{cl\,ran}(e^{i\theta} T - w)$, which implies that
\beq\label{Utheta-1}
U_\theta \bigcap_{w \in bpe(T)} \textup{cl\,ran}(T - w) \subseteq \bigcap_{w \in bpe(T)} \textup{cl\,ran}(e^{i\theta} T - w).
\eeq
Since $bpe(T)$ has circular symmetry (by Proposition \ref{circular symmetry}), we have $e^{i\theta} bpe(T) = bpe(T)$. This gives that
\beq\label{ran-M}
\bigcap_{w \in bpe(T)} \textup{cl\,ran}(e^{i\theta} T - w) &=& \bigcap_{u \in bpe(T)} \textup{cl\,ran}(e^{i\theta} T - e^{i\theta} u) \notag \\
&=& \bigcap_{u \in bpe(T)} \textup{cl\,ran}(T - u) = \mathcal M^\bot.
\eeq
The identities \eqref{Utheta-1} and \eqref{ran-M} together show that $U_\theta \mathcal M^\bot \subseteq \mathcal M^\bot$. Similarly, using $U^*_\theta T = e^{-i\theta} T U^*_\theta$, we get that $U^*_\theta \mathcal M^\bot \subseteq \mathcal M^\bot$. Thus $U_\theta|_{\mathcal M^\bot}$ is a unitary on $\mathcal M^\bot$, and for all $x \in \mathcal M^\bot$, we have $U_\theta T x = e^{i\theta} T U_\theta x$. Hence, $T|_{\mathcal M^\bot}$ is circular. This establishes (iii). The conclusion in (iv) follows easily from the following observation:
\beqn
\bigcap_{n \in \mathbb N} T^n(\mathcal M^\bot) \subseteq \bigcap_{n \in \mathbb N} T^n(\mathcal H) = \{0\}.
\eeqn

For the proof of (v), let $x \in \ker T^*$. Then we get
\beqn
0 = T^* x = T^* P_{\mathcal M} x + T^* P_{\mathcal M^\bot} x,
\eeqn
which, in turn, implies that $0 = P_{\mathcal M^\bot} T^* x = P_{\mathcal M^\bot} T^* P_{\mathcal M^\bot} x$. Thus $P_{\mathcal M^\bot} x \in \ker(T|_{\mathcal M^\bot})^*$. This completes the proof.
\end{proof}

The following corollary is an immediate fallout of the preceding proposition and the fact that the spectrum of an analytic operator is connected \cite[Lemma 5.2]{CT}.

\begin{corollary}
Let $T$ be an analytic and circular operator on a complex separable Hilbert space $\mathcal H$ with the wandering subspace property. Let $\mathcal M := \bigvee_{w \in bpe(T)} E^*_w(\ker T^*)$. Then the spectrum of $T|_{\mathcal M^\bot}$ is a closed disc centred at the origin with radius being equal to the spectral radius of $T|_{\mathcal M^\bot}$. 
\end{corollary}

As mentioned earlier, the investigation of the equality $\bigvee_{w \in \sigma(T)} \ker(T^* - \overline{w})=\mathcal H$ is a pivotal part in the study of the function theoretic behaviour of $T$. Clancey and Rogers proved that the aforementioned equality holds for a completely non-normal hyponormal operator whose approximate point spectrum has area measure zero \cite[Theorem 2]{CR}. We give below another sufficient condition which ensures that the said equality holds for a left-invertible operator with the wandering subspace property. This condition was also derived in \cite[Proposition 2.7]{S} but here we provide a different proof. Before proceeding towards the result, we need the following prerequisites from the local spectral theory.

Let $X$ be a Banach space and $T$ be a bounded linear operator on $X$. For $x\in X$, the local resolvent of $T$ at $x$, denoted by $\rho_T(x)$, is the union of all open subsets $G$ of $\mathbb{C}$ for which there is an analytic function $f:G\longrightarrow X$ satisfying $(\lambda-T)f(\lambda)=x$ for all $\lambda\in G$. The complement of $\rho_T(x)$ is called the local spectrum of $T$ at $x$ and is denoted by $\sigma_T(x)$. The following theorem will be needed in the subsequent proposition.

\begin{theorem}\cite[Theorem 1.6.3]{LN}\label{local}
Let $T$ be a bounded linear operator on a complex Banach space $X$. If $T$ is analytic, then $0 \in \sigma_T(x)$ for all non-zero $x \in X$.
\end{theorem} 

\begin{proposition}\label{prop2.10}
Let $T$ be a left-invertible operator on a complex separable Hilbert space $\mathcal H$ and let $T'$ denote the Cauchy dual of $T$. Suppose that $T$ has the wandering subspace property. If $T$ is analytic, then $\bigvee_{w \in \mathbb D(0, r(T')^{-1})} E^*_w(\ker T^*) = \mathcal H$.
\end{proposition}

\begin{proof}
Suppose that $T$ is analytic. Note that $(T^*T)^{-1} = T'^*T' \leqslant \|T'\|^2 I$. This gives that $T-w$ is left-invertible for all $w \in \mathbb D(0,\|T'\|^{-1})$. In view of Corollary \ref{bpe-incl-cor}, set 
\beqn
\mathcal M := \bigvee_{w \in \mathbb D(0, \|T'\|^{-1})} E^*_w(\ker T^*).
\eeqn
Let $x \in \mathcal M^\bot$. Then $x \in \ran\,(T-w)$ for all $w \in \mathbb D(0, \|T'\|^{-1})$, and hence, $x = (T-w) x(w)$ for all $w \in \mathbb D(0, \|T'\|^{-1})$. Consequently, we get
\beqn
T'^*x = T'^* (T-w) x(w) = (I-w T'^*) x(w) \mbox{ for all } w \in \mathbb D(0, \|T'\|^{-1}).
\eeqn
Since $(I-w T'^*)$ is invertible for all $w \in \mathbb D(0, \|T'\|^{-1})$, we have
\beqn
x(w) = (I-w T'^*)^{-1} T'^*x = \sum_{n \in \mathbb N} w^n T'^{*n+1} x \mbox{ for all } w \in \mathbb D(0, \|T'\|^{-1}).
\eeqn
This shows that the map $w \mapsto x(w)$ is analytic on $\mathbb D(0, \|T'\|^{-1})$ which satisfies that $(T-w) x(w) = x$ for all $w \in \mathbb D(0, \|T'\|^{-1})$. Thus $\mathbb D(0, \|T'\|^{-1}) \subseteq \rho_T(x)$. On the other hand, by Theorem \ref{local}, we have that $0 \in \sigma_T(x)$ for all non-zero $x$. Hence, we must have $x=0$. Thus, we get
\beqn
\mathcal H = \mathcal M \subseteq \bigvee_{w \in \mathbb D(0, r(T')^{-1})} E^*_w(\ker T^*).
\eeqn
This completes the proof.
\end{proof}

\section{Analytic bounded point evaluation}

Let $T$ be a bounded linear operator on a complex separable Hilbert space $\mathcal H$ with the wandering subspace property. If $bpe(T)$ has non-empty interior, then define the map $U : \{p(T) : p \in \mathcal P(bpe(T), \ker T^*)\} \rar \mathcal P(bpe(T), \ker T^*)$ by
\beq\label{U}
U(p(T)) = p, \quad p \in \mathcal P(bpe(T), \ker T^*),
\eeq
where $\mathcal P(bpe(T), \ker T^*)$ is the vector space of all $\ker T^*$-valued polynomials defined on $bpe(T)$. It is easy to see that $U$ is a well-defined bijective linear map. Thus, we can define the inner product on $\mathcal P(bpe(T), \ker T^*)$ by 
\beq\label{inp}
\inp{p}{q} := \inp{p(T)}{q(T)}_{\mathcal H}, \quad p, q \in \mathcal P(bpe(T), \ker T^*).
\eeq
Let $\mathscr H$ be the completion of $\mathcal P(bpe(T), \ker T^*)$ with respect to the inner product given by \eqref{inp}. Since $\{p(T) : p \in \mathcal P(bpe(T), \ker T^*)\}$ is a dense subspace of $\mathcal H$ and $U$ is unitary on this subspace, $U$ extends to a unitary from $\mathcal H$ onto $\mathscr H$. We denote this extension of $U$ by $U$ itself. We say that $w \in \mathbb C$ is a {\it bounded point evaluation on} $\mathscr H$ if there exists a positive constant $c$ such that $\|p(w)\|_{\mathcal H} \leqslant c \|p\|_{\mathscr H}$ for all $p \in \mathcal P(bpe(T), \ker T^*)$. If $w$ is a bounded point evaluation on $\mathscr H$, then the evaluation map $\mathscr E_w : \mathscr H \rar \ker T^*$ is continuous. In the following proposition, besides recording the properties of $U$, we also show that the notion of bounded point evaluation defined for $T$ agrees with the usual notion of bounded point evaluation on $\mathscr H$. 

\begin{remark}
The Corollary \ref{bpe-incl-cor} shows that the interior of $bpe(T)$ is non-empty for every operator $T$ in the class of left-invertible operators with the wandering subspace property. 
\end{remark}

\begin{proposition}\label{U-propo}
Let $T$ be a bounded linear operator on a complex separable Hilbert space $\mathcal H$ with the wandering subspace property. Suppose that $bpe(T)$ has non-empty interior. Let $\mathscr H$ be the Hilbert space as described above and $U : \mathcal H \rar \mathscr H$ be the unitary given by \eqref{U}. Then the following statements hold:
\begin{enumerate}
\item[$(i)$] $U$ maps $\ker T^*$ onto the subspace of constant polynomials in $\mathscr H$ and $U|_{\ker T^*}$ is the identity map.
\item[$(ii)$] The set of all bounded point evaluations on $\mathscr H$ is equal to $bpe(T)$.
\item[$(iii)$] For $w \in bpe(T)$, if $E_w$ and $\mathscr E_w$ are respective evaluation maps on $\mathcal H$ and $\mathscr H$, then $U E_w = \mathscr E_w U$.
\end{enumerate}
\end{proposition}

\begin{proof}
The proofs of (i) and (ii) are immediate from the definition of $U$ and that of bounded point evaluation on $\mathscr H$. To see (iii), let $w \in bpe(T)$ and $x \in \mathcal H$. Then there exists a sequence of polynomials $\{p_n\}_{n \geqslant 1}$ in $\mathcal P(bpe(T), \ker T^*)$ such that $p_n(T) \rar x$ as $n \rar \infty$. This, together with (i), gives that
\beqn
p_n(w) = U p_n(w)  \rar UE_w x \mbox{ as } n \rar \infty.
\eeqn
Further, we also have
\beqn
p_n(w) = \mathscr E_w p_n = \mathscr E_w U p_n(T) \rar \mathscr E_w U x \mbox{ as } n \rar \infty.
\eeqn
Thus we get that $U E_w x = \mathscr E_w Ux$. This completes the proof.
\end{proof}

\begin{theorem}\label{Mz}
Let $T$ be a bounded linear operator on a complex separable Hilbert space $\mathcal H$ with the wandering subspace property. Suppose that $bpe(T)$ has non-empty interior. Let $\mathscr H$ be the Hilbert space as described after \eqref{inp}. Then the following statements are true:
\begin{enumerate}
\item[$(i)$] The operator $\mathscr M_z$ of multiplication by the coordinate function defined on the polynomials in $\mathscr H$ extends continuously on $\mathscr H$. If the continuous extension of $\mathscr M_z$ is denoted by $\mathscr M_z$ itself, then $\mathscr M_z U = UT$, where $U : \mathcal H \rar \mathscr H$ is the unitary operator as defined by \eqref{U}.
\item[$(ii)$] The elements of $\mathscr H$ are well-defined $\ker T^*$-valued functions on $bpe(T)$ if and only if $\bigvee_{w \in bpe(T)} E^*_w(\ker T^*) = \mathcal H$.
\end{enumerate}
\end{theorem}

\begin{proof}
The proof of (i) follows by imitating the arguments of the proof of \cite[Theorem 3.1]{T}. For the proof of (ii), assume that $\bigvee_{w \in bpe(T)} E^*_w(\ker T^*) = \mathcal H$. Let $f \in \mathscr H$ and $f(w) = 0$ for all $w \in bpe(T)$. Then there exists a vector $x \in \mathcal H$ such that $U x = f$. Using Proposition \ref{U-propo}, we get that
\beqn
0 = f(w) = \mathscr E_w U x = U E_w x = E_w x \mbox{ for all } w \in bpe(T).
\eeqn
Thus we get
\beqn
x \in \bigcap_{w \in bpe(T)} \ker E_w = \Big(\bigvee_{w \in bpe(T)} E^*_w(\ker T^*)\Big)^\bot = \mathcal H^\bot = \{0\}.
\eeqn 
This gives that $f = 0$. The proof of the converse is obvious. 
\end{proof}

Let $\mathscr H$ be the Hilbert space as described after \eqref{inp}. Then for $f \in \mathscr H$ and $w \in bpe(T)$, we define 
\beqn
\hat f (w) := \mathscr E_w (f).
\eeqn
Thus $\hat f$ is a well-defined $\ker T^*$-valued function on $bpe(T)$. The largest open subset of $bpe(T)$ on which $\hat f$ is analytic (or holomorphic) for all $f \in \mathscr H$, is said to be the  set of all {\it analytic bounded point evaluations for} $T$, and is denoted by $abpe(T)$. It follows from Proposition \ref{U-propo} and  Theorem \ref{Mz}(ii) that $f \in \mathscr H$ can be identified with $\hat f$ if and only if $\bigvee_{w \in bpe(T)} \mathscr E^*_w(\ker T^*) = \mathscr H$. In this case, $\mathscr H$ turns out to be a reproducing kernel Hilbert space of $\ker T^*$-valued functions defined on $bpe(T)$ with the reproducing kernel $\kappa : bpe(T) \times bpe(T) \rar \mathcal B(\ker T^*)$ given by $\kappa(z,w) = \mathscr E_z \mathscr E_w^*,\  z, w \in bpe(T)$, \cite[Chapter 6]{PR}. 

The following theorem gives several equivalent characterizations of $abpe(T)$. Its proof can be easily obtained by following the arguments of the proof of \cite[Theorem 3.4]{T} along with the help of Theorem \ref{bpe-thm}, Proposition \ref{bpe-prop} and Theorem \ref{Mz}.

\begin{theorem}\label{apbe-thm-old}
Let $T$ be a bounded linear operator on a complex separable Hilbert space $\mathcal H$ with the wandering subspace property. Suppose that $bpe(T)$ has non-empty interior. Let $\mathscr H$ be the Hilbert space as described after \eqref{inp}. Then the following statements hold:
\begin{enumerate}
\item[$(i)$] Suppose that $\mathcal O$ is an open set contained in $bpe(T)$. Then $\mathcal O \subseteq abpe(T)$ if and only if the map $w \mapsto \mathscr E_w^*$ of $\mathcal O$ into $\mathcal B(\ker T^*, \mathscr H)$ is bounded on compact subsets of $\mathcal O$. 
\item[$(ii)$] A point $w \in abpe(T)$ if and only if there exists a neighbourhood $\mathcal O \subseteq bpe(T)$ of $w$ such that the map $\lambda \mapsto \mathscr E_\lambda^* x$ of $\mathcal O$ into $\mathscr H$ is conjugate analytic for all $x \in \ker T^*$.
\item[$(iii)$] Let $\{x_n : 1 \leqslant n \leqslant \dim \ker T^*\}$ be an orthonormal basis of $\ker T^*$. A point $w \in abpe(T)$ if and only if there exist a neighbourhood $\mathcal O $ of $w$ and conjugate analytic maps $\psi_j : \mathcal O \rar \mathscr H$, $j=1, \ldots, \dim\ker T^*$, such that $\psi_j(\lambda) \in \ker(\mathscr M_{z}^*-  \overline{\lambda})$ and $\inp{\psi_j(\lambda)}{x_i} = \delta_{ij}$ for all $\lambda \in \mathcal O$ and $i, j = 1, \ldots, \dim \ker T^*$.
\end{enumerate}
\end{theorem}

We give below a complete description of $abpe(T)$ for a left-invertible operator $T$ with the wandering subspace property. Its proof is immediate from Theorem \ref{bpe-left-thm} and Theorem \ref{apbe-thm-old}(i).
 
\begin{theorem}\label{abpe-left-thm}
Let $T$ be a left-invertible operator on a complex separable Hilbert space $\mathcal H$ with the wandering subspace property. Let $T'$ denote the Cauchy dual of $T$. For each $w \in \mathbb C$ and $n \in \mathbb N$, set
\beqn
\mathcal M_n := \bigvee_{k=0}^n T^k(\ker T^*) \quad \mbox{and} \quad S_n(w) := \sum_{k=0}^n \overline{w}^k T'^k.
\eeqn
Then $u \in abpe(T)$ if and only if there exists an open set $\mathcal O \subseteq bpe(T)$ containing $u$ such that the map $w \mapsto \displaystyle\sup_{n \in \mathbb N}\|P_{\mathcal M_n} S_n(w)|_{ker T^*}\|$ of $\mathcal O$ into $[0, \infty)$ is bounded on compact subsets of $\mathcal O$.
\end{theorem}

The proof of the following corollary is immediate from Corollary \ref{bpe-incl-cor} and the preceding theorem.

\begin{corollary}\label{cor3.6}
Let $T$ be a left-invertible operator on a complex separable Hilbert space $\mathcal H$ with the wandering subspace property. Let $T'$ denote the Cauchy dual of $T$. Then the open disc $\mathbb D(0, r)$ is contained in the $abpe(T)$, where $\displaystyle r := \inf_{x \in \ker T^*} \frac{1}{r_{T'}(x)}$. In particular, $\mathbb D\big(0, r(T')^{-1}\big) \subseteq abpe(T)$. Moreover, for each $w \in \mathbb D(0, r)$, we have $\mathscr E_w^* x = \sum_{n \in \mathbb N} \overline{w}^n T'^n x$ for all $x \in \ker T^*$.
\end{corollary}

It can be easily seen from Theorem \ref{Mz}, Proposition \ref{prop2.10} and the preceding corollary that if $T$ is a left-invertible analytic operator with the wandering subspace property, then $T$ is unitarily equivalent to $\mathscr M_z$ on the Hilbert space of $\ker T^*$-valued holomorphic functions defined on $\mathbb D\big(0, r(T')^{-1}\big)$ with the reproducing kernel given by
\beqn
\kappa(z,w) = \mathscr E_z \mathscr E_w^* = (I-z T'^*)^{-1}(I-\overline{w} T')^{-1} \mbox{ for all } z, w \in \mathbb D\big(0, r(T')^{-1}\big).
\eeqn
This is precisely Shimorin's model \cite{S} for a left-invertible analytic operator with the wandering subspace property. We conclude this paper with a couple of examples of left-invertible operators $T$ with the wandering subspace property for which $\mathbb D\big(0, r(T')^{-1}\big) \subsetneqq bpe(T)$.
 
The following example is motivated from \cite[Example 2]{CT} and \cite[Example 3.1]{ACT}.

\begin{example}
Consider the operator $T$ on $\ell^2(\mathbb N)$ defined as $T e_0 = \lambda_0 e_0 + \lambda_1 e_1$ and $T e_n = \lambda_{n+1} e_{n+1}$ for all  $n \geqslant 1$, where $\{e_n : n \in \mathbb N\}$ is the standard orthonormal basis of $\ell^2(\mathbb N)$ and the weights $\lambda_n$ are described as follows:  \\
$$\lambda_0 =\lambda_1=\lambda_2=\lambda_3=\lambda_4=1,\ \mbox{and}\ \lambda_k=\begin{cases}
\frac{1}{2}, & \mbox{if}\ 2^n+1 \leq k \leq 3.2^{n-1},\ n \geq 2,\\
1, & \mbox{otherwise.}
\end{cases}$$
The operator $T$ is nothing but a weighted composition operator on $\ell^2(\mathbb N)$ with the symbol $\phi: \mathbb N \rar \mathbb N$ defined as $\phi(0) = 0$ and $\phi(n) = n-1$ for all $n \geqslant 1$. It can also be realized as a weighted shift on the following directed graph:
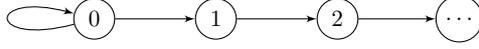
\begin{figure}[H]
\begin{tikzpicture}[scale=.8, transform shape]
\tikzset{vertex/.style = {shape=circle,draw, minimum size=1em}}
\tikzset{every loop/.style = {min distance = 15mm,  looseness = 15}}
\tikzset{edge/.style = {->,> = latex'}}
% vertices
\node[vertex] (a) at  (0,0) {$0$};
%\node[vertex] (b) at  (2,-1) {$1$};
\node[vertex] (c) at  (2,0) {$1$};
%\node[vertex] (d) at  (4, -1) {$3$};
\node[vertex] (e) at  (4, 0) {$2$};
%\node[vertex] (f) at  (6, -1) {$\ldots$};
\node[vertex] (g) at  (6, 0) {$\ldots$};
%\node[vertex] (h) at  (8, -1) {$\ldots$};

%\tikzset{vertex/.style = {shape=circle,draw, blue, minimum size=1em}}
%\node[vertex] (i) at  (8, 1) {$\ldots$};

%edges
\draw[edge] (a) to (c);
%\draw[edge] (a) to (c);
\draw[edge] (c) to (e);
%\draw[edge] (c) to (e);
\draw[edge] (e) to (g);
%\draw[edge] (e) to (g);
%\draw[edge] (f) to (h);
%\draw[edge] (g) to (i);

%\draw[edge] (a) to (n);
%\draw[edge] (n) to (l);
%\draw[edge] (l) to (k);
%\draw[edge] (k) to (j);
%\draw[edge] (j) to (m);
%\draw[edge] (m) to (o);
\draw[loop left][edge] (a) to (a);

%\def\Radius{.5}
%\draw (a) arc(0:360:\Radius) -- cycle;

%\draw[edge] (h) to (j);
%\draw[edge] (i) to (k);
%\draw[orange,rotate=45,shift={(3 cm,5 cm)}] \draw[edge] (e) to (g);
\end{tikzpicture}
\end{figure}
\noindent
Note that $\inf_{n\geq1} \lambda_n = \frac{1}{2}$. Thus $T$ is left-invertible and the action of the Cauchy dual $T'$ of $T$ is same as that of $T$ with weights $\{\lambda'_n : n \in \mathbb N\}$ described as follows:
\beqn
\lambda'_0 = \lambda'_1 = \frac{1}{2} \quad \mbox{and} \quad \lambda'_n = \frac{1}{\lambda_n} \mbox{ for all } n \geqslant 2.
\eeqn
Note that the series $\sum_{m \in \mathbb N} \frac{\lambda'_1 \cdots \lambda'_{m+1}}{(\lambda'_0)^{m+1}}$ diverges. Therefore, by \cite[Theorem 2.1]{ACT}, $T'$ is analytic, and hence, by \cite[Proposition 2.7]{S}, $T$ has the wandering subspace property. Further, $\ker T^* = \mbox{span}\{e_0-e_1\}$. Also, observe that for any $n \geqslant 1$, the total number of $2$'s occurring in first $2^n$ places in $\{\lambda'_n : n \geqslant 1\}$ is equal to $2^{n-1}-2^{n-2}+2^{n-2}-2^{n-3}+\cdots+4-2=2^{n-1}-2$. Therefore, we get 
$$\lambda'_1 \lambda'_2 \cdots \lambda'_{2^n}=\frac{2^{2^{n-1}-2}}{2}=\frac{2^{2^{n-1}}}{8}.$$ 
Let $n$ be any positive integer. Then there are unique non-negative integers $m_n$ and $k_n$ such that $n=2^{m_n}+k_n$ with $0 \leqslant k_n < 2^{m_n}$.  Thus, we have 
\beq\label{prod-eq-ex}
\lambda'_1 \lambda'_2 \cdots \lambda'_{n} = \frac{2^{2^{m_n-1}-2+\alpha_{k_n}}}{2} = \frac{2^{2^{m_n-1}+\alpha_{k_n}}}{8},
\eeq
where
\beqn
\alpha_{k_n} = \begin{cases}
k_n & \mbox{ if } 0 \leqslant k_n \leqslant 2^{m_n - 1},\\
2^{m_n - 1} & \mbox{ if } 2^{m_n - 1} < k_n < 2^{m_n}.
\end{cases}
\eeqn 
Also, observe that if $0 \leqslant k_n < 2^{m_n - 1}$, then
\beqn
\frac{2^{m_n - 1} + \alpha_{k_n}}{2^{m_n} + k_n} = \frac{1}{2} \frac{2^{m_n} + 2k_n}{2^{m_n} + k_n} = \frac{1}{2} + \frac{1}{2} \frac{1}{\frac{2^{m_n}}{k_n} + 1} < \frac{1}{2} + \frac{1}{2} \frac{1}{3} = \frac{2}{3}.
\eeqn
Further, for $2^{m_n - 1} \leqslant k_n < 2^{m_n}$, we have
\beqn
\frac{2^{m_n - 1} + \alpha_{k_n}}{2^{m_n} + k_n} = \frac{2^{m_n}}{2^{m_n} + k_n} \leqslant \frac{2^{m_n}}{3 \cdot 2^{m_n - 1}} = \frac{2}{3}. 
\eeqn 
Consequently, we get
\beq\label{2by3}
\frac{2^{m_n - 1} + \alpha_{k_n}}{2^{m_n} + k_n} \leqslant \frac{2}{3} \mbox{ for all } n \geqslant 1.
\eeq
Let $x = e_0-e_1$. Then for each $n \geqslant 1$, we get
\beqn
\|T'^n x\|^2 &=& \|T'^n e_0\|^2 + \|T'^n e_1\|^2\\
&=& |\lambda'_0|^{2n} + |\lambda'_0|^{2(n-1)} |\lambda'_1|^2 + \cdots +|\lambda'_1 \lambda'_2 \cdots \lambda'_{n}|^2 + |\lambda'_2 \lambda'_3 \cdots \lambda'_{n+1}|^2\\
&\overset{\eqref{prod-eq-ex}}\leqslant& (n+1) \Big(\frac{2^{2^{m_n-1}+\alpha_{k_n}}}{8}\Big)^2 + 16 \Big(\frac{2^{2^{m_n-1}+\alpha_{k_n}}}{8}\Big)^2 \\
&=& (n+17) \Big(\frac{2^{2^{m_n-1}+\alpha_{k_n}}}{8}\Big)^2.
\eeqn
This together with \eqref{2by3} implies that $\limsup_{n \rar \infty} \|T'^n x\|^\frac{1}{n} \leqslant 2^\frac{2}{3}$. On the other hand, it is not difficult to see that $r(T') = 2$. Thus, in view of Corollary \ref{cor3.6}, we have
\beqn
\mathbb D\big(0, r(T')^{-1}\big) \subsetneqq \mathbb D(0, 2^{-\frac{2}{3}}) \subseteq abpe(T)  \subseteq bpe(T).
\eeqn
\end{example}

The foregoing example suggests a way to construct a class of examples for which $\mathbb D\big(0, r(T')^{-1}\big) \subsetneqq abpe(T)  \subseteq bpe(T)$. For this, we refer the reader to \cite{JJS} for details about the weighted shift on directed trees.

\begin{example}
Consider the sequence $\{\lambda_n\}_{n \geqslant 1}$ of positive numbers with the following properties:
\begin{enumerate}
\item[(i)] $\lambda_n \leqslant 1$;
\item[(ii)] $\displaystyle\inf_{n \geqslant 1} \lambda_n > 0$; 
\item[(iii)] $\displaystyle\limsup_{n \rar \infty} (\lambda_1 \cdots \lambda_n)^{-\frac{1}{n}} < \lim_{n \rar \infty} \big(\sup_{m \in \mathbb N}(\lambda_{m+1} \cdots \lambda_{m+n})^{-1} \big)^\frac{1}{n}$.
\end{enumerate} 
One such example of a sequence satisfying above three conditions has been exhibited in the previous example. Let $k$ be any positive integer. Then consider the rooted directed tree $\mathscr T_{k,0}= (V, \mathcal E)$ described as follows (the reader is referred to \cite[Chapter 6]{JJS} for more details about such trees):
\beqn
V &=& \{(0,0), (m,n) : m, n \geqslant 1\}; \\
\mathcal E &=& \{\big((0,0), (1,j)\big) : 1\leqslant j \leqslant k\} \cup \{\big((m,n), (m+1,n)\big) : m, n \geqslant 1\}.
\eeqn
An example of $\mathscr T_{3,0}$ is depicted in the following figure.
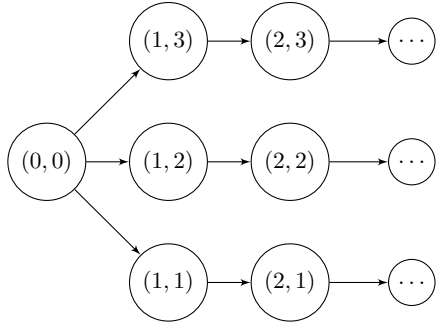
\begin{figure}[H]
\begin{tikzpicture}[scale=.8, transform shape]
\tikzset{vertex/.style = {shape=circle,draw, minimum size=1em}}
\tikzset{every loop/.style = {min distance = 15mm,  looseness = 15}}
\tikzset{edge/.style = {->,> = latex'}}
% vertices
\node[vertex] (a) at  (0,0) {$(0,0)$};
\node[vertex] (b) at  (2,-2) {$(1,1)$};
\node[vertex] (c) at  (2,0) {$(1,2)$};
\node[vertex] (d) at  (2, 2) {$(1,3)$};
\node[vertex] (e) at  (4, -2) {$(2,1)$};
\node[vertex] (f) at  (4, 0) {$(2,2)$};
\node[vertex] (g) at  (4, 2) {$(2,3)$};
\node[vertex] (h) at  (6, -2) {$\ldots$};
\node[vertex] (i) at  (6, 0) {$\ldots$};
\node[vertex] (j) at  (6, 2) {$\ldots$};
%\tikzset{vertex/.style = {shape=circle,draw, blue, minimum size=1em}}
%\node[vertex] (i) at  (8, 1) {$\ldots$};

%edges
\draw[edge] (a) to (b);
\draw[edge] (a) to (c);
\draw[edge] (a) to (d);
\draw[edge] (b) to (e);
\draw[edge] (c) to (f);
\draw[edge] (d) to (g);
\draw[edge] (e) to (h);
\draw[edge] (f) to (i);
\draw[edge] (g) to (j);

%\draw[edge] (a) to (n);
%\draw[edge] (n) to (l);
%\draw[edge] (l) to (k);
%\draw[edge] (k) to (j);
%\draw[edge] (j) to (m);
%\draw[edge] (m) to (o);
%\draw[loop left][edge] (a) to (a);

%\def\Radius{.5}
%\draw (a) arc(0:360:\Radius) -- cycle;

%\draw[edge] (h) to (j);
%\draw[edge] (i) to (k);
%\draw[orange,rotate=45,shift={(3 cm,5 cm)}] \draw[edge] (e) to (g);
\end{tikzpicture}
\end{figure}
\noindent
The weights on the vertices of the tree $\mathscr T_{k,0}$ are described as follows:
\beqn
\lambda_{(n,1)} = \lambda_n, \ n \geqslant 1 \quad \mbox{and} \quad \lambda_{(m,n)} = 1 \mbox{ for all } m \geqslant 1 \mbox{ and } n = 2, \ldots, k.
\eeqn
Consider the Hilbert space $\ell^2(V)$. Let $T : \ell^2(V) \rar \ell^2(V)$ be the weighted shift whose action on the basis vectors is defined as follows:
\beqn
T e_{(0,0)} = \sum_{n=1}^k \lambda_{(1,n)} e_{(1,n)} \quad \mbox{and} \quad T e_{(m,n)} = \lambda_{(m+1,n)} e_{(m+1,n)} \mbox{ for all } m,n \geqslant 1.
\eeqn
Then $T$ is left-invertible. That $T$ has the wandering subspace property follows from \cite[Proposition 1.3.4]{CPT}. The Cauchy dual $T'$ of $T$ is again a weighted shift on $\ell^2(V)$ with weights $\{\lambda'_{(m,n)} : m,n \geqslant 1\}$ given  by 
\beqn
\lambda'_{(1,n)} = \frac{\lambda_{(1,n)}}{\|T e_{(0,0)}\|^2}, \ 1 \leqslant n \leqslant k, \mbox{ and } \lambda'_{(m,n)} = \frac{1}{\lambda_{(m,n)}}, \ m \geqslant 2,\ n = 1, \ldots, k.
\eeqn
Further,
\beqn
\ker T^* = \mbox{span}\{e_{(0,0)}\} \oplus \big(\ell^2(\{(1,n): 1\ \leqslant n \leqslant k\}) \ominus \mbox{span}\{T e_{(0,0)}\} \big).
\eeqn
Routine calculations show that for each $x = \alpha e_{(0,0)} + y$ in $\ker T^*$ and for each $n \geqslant 1$,
\beqn
\|T'^n x\|^2 &=& |\alpha|^2 \|T'^n e_{(0,0)}\|^2 + \|T'^n y\|^2 \\
 &=& |\alpha|^2 \big((\lambda'_{(1,1)})^2 (\lambda_2 \ldots \lambda_n)^{-2} + k-1\big) + a (\lambda_2 \ldots \lambda_n)^{-2} + b
 \eeqn
 for some non-negative scalars $a,b$. Consequently, $\|T'^n x\|^2 = c (\lambda_1 \ldots \lambda_n)^{-2} + d \ \mbox{ for some non-negative scalars } c, d$. This gives that for all $x \in \ker T^*$,  
 \beqn
r_{T'}(x) &=& \limsup_{n \rar \infty} \|T'^n x\|^\frac{1}{n} \leqslant  \limsup_{n \rar \infty} (\lambda_1 \cdots \lambda_n)^{-\frac{1}{n}} \\
&<& \lim_{n \rar \infty} \big(\sup_{m \in \mathbb N}(\lambda_{m+1} \cdots \lambda_{m+n})^{-1} \big)^\frac{1}{n} \leqslant r(T').
\eeqn
Let $\displaystyle r := \inf_{x \in \ker T^*} \frac{1}{r_{T'}(x)}$. Then by Corollary \ref{cor3.6}, we get 
\beqn
 \mathbb D\big(0, r(T')^{-1}\big) \subsetneqq \mathbb D(0, r) \subseteq abpe(T) \subseteq bpe(T).
\eeqn
\end{example}

\medskip \textit{Acknowledgment}. \
The author conveys his sincere thanks to Md. Ramiz Reza for his several helpful suggestions.

{}


\begin{thebibliography}{}

\bibitem{ACT}
A. Anand, S. Chavan and S. Trivedi, Analytic $m$-isometries without the wandering subspace property, {\it Proc. Amer. Math. Soc.} {\bf 148} (2020), 2129-2142.

\bibitem{AHHK}
W. Arveson, D. Hadwin, T. Hoover and E. Kymala, Circular operators, {\it Indiana Univ. Math. J.} {\bf 33} (1984), 583-595.

\bibitem{BCZ}
A. Bourhim, C. Chidume and E. Zerouali, Bounded point evaluations for cyclic operators and local spectra, {\it Proc. Amer. Math. Soc.} {\bf 130} (2002), 543-548.

\bibitem{CPT}
S. Chavan, D. Pradhan, S. Trivedi, Multishifts on directed Cartesian product of rooted directed trees, {\it Diss. Math.} {\bf 527} (2017) 1–102.

\bibitem{CT}
S. Chavan and S. Trivedi, An analytic model for left-invertible
weighted shifts on directed trees, {\it J. London Math. Soc.} {\bf 94} (2016), 253-279.

\bibitem{CR}
K. Clancey, and D. Rogers, Cyclic vectors and seminormal operators, {\it Indiana Univ. Math. J.} {\bf 27} (1978), 689-696. 

\bibitem{G}
R. Gellar, Circularly symmetric normal and subnormal operators, {\it J. Analyse Math.} {\bf 32} (1977), 93-117.

\bibitem{GMZ}
M. El Guendafi, M. Mbekhta and E. Zerouali, Bounded point evaluations for multicyclic operators, {\it Topological algebras, their applications, and related topics}, 199-217, {\it Banach Center Publ.} {\bf 67}, {\it Polish Acad. Sci. Inst. Math., Warsaw}, 2005.

\bibitem{GKT}
R. Gupta, S. Kumar and S. Trivedi, Unitary equivalence of operator-valued multishifts, {\it J. Math. Anal. Appl.}, {\bf 487} (2020), 23 pp.

\bibitem{JJS}
Z. Jab{\l}o\'nski, Il Bong Jung, and J. Stochel, Weighted shifts on directed
trees, {\it  Mem. Amer. Math. Soc.} {\bf 216} (2012), no. 1017, viii+106.

\bibitem{LN}
K. Laursen, and M. Neumann,  {\it An introduction to local spectral theory}, London Mathematical Society Monographs. New Series, 20. The Clarendon Press, Oxford University Press, New York, 2000. xii+591 pp.

\bibitem{MOZ}
M. Mbekhta, N. Ourchane and E. Zerouali, The interior of bounded point evaluations for rationally cyclic operators,  {\it Mediterr. J. Math.} {\bf 13} (2016), 1981-1996.

\bibitem{MMN}
T. Miller, V. Miller and M. Neumann, Analytic bounded point evaluations for rationally cyclic operators on Banach spaces, {\it Integral Equations Operator Theory} {\bf 51} (2005), 257-274.

\bibitem{PR}  V. Paulsen and M. Raghupathi, {\it An Introduction to the Theory of Reproducing Kernel Hilbert Spaces}, Cambridge Studies in Advanced Mathematics, 152. Cambridge University Press, Cambridge, 2016.

\bibitem{S}
S. Shimorin, Wold-type decompositions and wandering subspaces for operators
close to isometries, {\it J. Reine Angew. Math.} {\bf 531} (2001), 147-189.

\bibitem{Tr}
T. Trent, $H^2(\mu)$ spaces and bounded point evaluations, {\it Pacific J. Math.} {\bf 80} (1979), 279-292.

\bibitem{T}
S. Trivedi, Bounded point evaluation for a finitely multicyclic commuting tuple of operators, {\it Bull. Sci. Math.} {\bf 162} (2020), 20 pp. 

\bibitem{W}
L. Williams, Bounded point evaluations and the local spectra of cyclic hyponormal operators, {\it Dynam. Systems Appl.} {\bf 3} (1994), 103-112.

\end{thebibliography}
\end{document}